\documentclass[11pt]{article}
\usepackage[all]{xy}
\usepackage{amsmath}
\usepackage{amsfonts}
\usepackage{amssymb}
\usepackage{amscd}
\usepackage{amsthm}
\usepackage{latexsym}
\usepackage{amsbsy}

\def\0D{\Delta^{(0)}}
\def\1D{\Delta^{(1)}}

\newtheorem{theorem}{Theorem}[section]
\newtheorem{remark}[theorem]{Remark}
\newtheorem{proposition}[theorem]{Proposition}
\newtheorem{lemma}[theorem]{Lemma}
\newtheorem{corollary}[theorem]{Corollary}

\newtheorem{example}[theorem]{Example}
\newtheorem{definition}[theorem]{Definition}

\def\build#1_#2^#3{\mathrel{
\mathop{\kern 0pt#1}\limits_{#2}^{#3}}}

\numberwithin{equation}{section}

\def\a{\alpha}
\def\b{\beta}
\def\c{\chi}

\def\ot{\otimes}
\def\part{\partial}

\def\ra{\rightarrow}

\def\text{\hbox}

\def\ot{\otimes}

\def\ra{\rightarrow}

\def\Hom{\mathop{\rm Hom}\nolimits}

\def\Id{\mathop{\rm Id}\nolimits}

\def\build#1_#2^#3{\mathrel{
\mathop{\kern 0pt#1}\limits_{#2}^{#3}}}

\numberwithin{equation}{section}
\newcommand{\comment}[1]{\relax}

\begin{document}
\title{On Antipodes Of Hom-Hopf algebras}
\author {Mohammad Hassanzadeh}
\date{University of Windsor\\
Windsor, Ontario, Canada\\
mhassan@uwindsor.ca
}
\maketitle
\begin{abstract}
In the recent definition of  Hom-Hopf algebras  the antipode $S$ is the relative Hom-inverse of the  identity map with respect to the convolution product.
We observe  that some fundamental properties of the antipode of Hopf algebras and Hom-Hopf algebras, with the original definition, do not hold generally  in the new setting.
We show that  the antipode is a relative Hom-anti algebra and a relative anti-coalgebra morphism. It is also  relative Hom-unital, and  relative Hom-counital. Furthermore if the twisting maps of multiplications and comultiplications are invertible then $S$ is an anti-algebra and an anti-coalgebra map.
We show that any Hom-bialgebra map between two Hom-Hopf algebras is  a relative Hom-morphism of  Hom-Hopf alegbras. Specially if the corresponding  twisting maps are all invertible then it is a Hom-Hopf algebra  map.
If the Hom-Hopf algebra is commutative or cocommutative we observe  that $S^2$ is equal to the identity map in some sense.  At the end we study the images of  primitive and group-like elements under the antipode.

\end{abstract}

\section{ Introduction}

The examples  of Hom-Lie algebras were first appeared in $q$-deformations of
algebras of vector fields, such as Witt and Virasoro algebras \cite{as}, \cite{ckl}, \cite{cz}. The concept of Hom-Lie algebras generalizes the one for Lie algebras where the Jocobi identity is twisted by a homomorphism \cite{hls}, \cite{ls}.
Hom-associative algebras were introduced  and studied in \cite{ms1}. Moreover Hom-coalgebras and Hom-bialgebras were studied in \cite{ms2}, \cite{ms3}, \cite{ya2}, \cite{ya3}, \cite{ya4}. In the last years, many  classical algebraic concepts have
been extended to the framework of Hom-structures. For examples see \cite{hls}, \cite{gw}, \cite{pss}, \cite{hss}, \cite{aem}, \cite{gmmp}, \cite{gr}, \cite{cq}, \cite{cs}, \cite{zz}.
The Hom-Hopf algebras first introduced in \cite{ms2} and \cite{ms3}. In these works they defined a Hom-Hopf algebra $H$ to be a Hom-bialgebra $(H, \mu, \a, \eta, \Delta, \b, \varepsilon)$, endowed with a map $S: H\longrightarrow H$, where it is the inverse of the identity map $\Id_H$ with respect to the convolution product $\star$, i.e,

$$ S \star \Id = \Id \star S= \eta\circ \varepsilon.$$

 This definition of antipode is the same as the one for Hopf algebras. The universal enveloping algebra of a Hom-Lie algebra introduced in \cite{ya4}.  It has been shown that it is a Hom-bialgebra. However
it is  not  a Hom-Hopf algebra in the sense of \cite{ms2}, since it is shown in \cite{lmt} that  the antipode is not an inverse of the identity map with respect to the  the convolution product. This was a motivation to change definition of the antipode  such that a Hom-Hopf algebra is a Hom-bialgebra which  satisfies a weakened condition. For every $h\in H$ there exists $k\in \mathbb{N}$ satisfying the weakened condition
$$  \a^k(S \star \Id )(h)= \a^k(\Id \star S)(h)= \eta\circ \varepsilon(h). $$ This naturally suggests to change the definition of invertible elements of a Hom-algebra $A$ as being elements $a\in A$ such that there exists $b\in A$  and $k \in \mathbb{N}$ where $ \a^k( ab) = \a^k(ba) =1_A.$
This means   the antipode is the relative Hom-inverse of the identity map.
 In this paper we study this  recent notion of Hom-Hopf algebras. More precisely by Definition \ref{Hom-Hopf}  a Hom-Hopf algebra in the new setting is a Hom-bialgebra endowed with a unital, counital, anti-algebra and anti-coalgebra  map $S: H\longrightarrow H$ which is relative Hom-inverse of  the identity map $\Id_H$, and it commutes with $\a$.
 The Hom-Hopf algebras in  Examples \ref{Sweedler} and \ref{2-dimensional Hopf} satisfy the conditions of both definitions.
   The set of group-like elements and primitive elements are important to study Hopf type objects.  The group-like elements gives a relation between Hom-Hopf algebras and Hom-groups  while primitive elements connects Hom-Hopf algebras to Hom-Lie algebras.
The authors in \cite{lmt} showed that the set of group-like elements in  a Hom-Hopf algebra is a Hom-group where the inverse elements are given by the antipode. In Example \ref{polynomial}, we introduce a Hom-bialgebra containing a group-like element which does not have any inverse. Therefore it does not  have any antipode or Hom-Hopf algebra structure.
 The main aim  of this paper is to find out  if one removes the important conditions unitality, counitality, anti-algebra map, anti coalgebra map, and $S\circ \a=\a\circ S$, from  Definition \ref{Hom-Hopf},
 and only sticks with the relative Hom-invertibility condition of $S$, then how much of these properties can be recovered and what are the  other properties of the antipode.
 To investigate this, we consider a Hom-bialgebra endowed with a map $S$  which is a  relative Hom-inverse of the identity map.
  First we need to find out the relations between relative Hom-inverse elements with respect to the convolution product in Proposition \ref{important}. In Propositions \ref{relative anti algebra} and \ref{relative anti coalgebra}, we show that  the antipode is a relative Hom-anti-algebra and a relative Hom-anti-coalgebra morphism.
 it is also shown in  Propositions \ref{hom unitality prop} and \ref{relative counitality} that  the antipode is  relative Hom-unital and  relative Hom-counital.
Furthermore if the twisting maps $\a$ and $\b$  are invertible then $S$ is an anti-algebra and an anti-coalgebra map.
Then in Proposition \ref{Hopf map} we prove that any Hom-bialgebra map between two Hom-Hopf algebras is  a relative Hom-morphism of  Hom-Hopf alegbras.  By Corollary \ref{hopf morphism}, if the corresponding  twisting maps are all invertible then it is a Hom-Hopf algebra  map. Furthermore we observe that if $\a=\b$ then $S$ commutes with powers of $\a$. Later we study $S^2$ for  commutative  and  cocommutative Hom-Hopf algebras. In these cases we prove that $S^2$ is equal to the identity map in some sense. If $\a$ and $\b$ are invertible then $S^2=\Id$. At the end we study the images of  primitive and group-like elements under the antipode.

\bigskip

\textbf{Notations}:In this paper all (Hom)-algebras, (Hom)-colagebras, (Hom)-bialgebras and (Hom)-Hopf algebras are defined on a field $\mathbb{K}$.
All tensor products $\ot$ are on a field $\mathbb{K}$. We denote the set of natural numbers by $\mathbb{N}$.

\tableofcontents
\section{ Hom-Hopf algebras}

In this section we recall the basics of Hom-algebras, Hom-coalgebras, Hom-bialgebras and Hom-Hopf algebras. To understand these structures we introduce some  examples.

   By \cite{ms1}, a Hom-associative algebra  $A$ over a field $\mathbb{K}$ is a $\mathbb{K}$-vector space with a bilinear map $m: A\ot A\longrightarrow A$, called multiplication, and  a linear homomorphism
  $\alpha: A\longrightarrow A$ satisfying the Hom-associativity  condition
  $$m \circ (m \ot \alpha)= m \circ ( \alpha \ot m).$$  In terms of elements $a,b,c\in A$, this can be written as $\a(a)(bc) = (ab)\a(c)$.
  The Hom-associativity property in terms of a commutative diagram is

 $$
\xymatrix{
A\ot A\ot A \ar[r]^{m\ot \a} \ar[d]_{\a\ot m} & A\ot A \ar[d]^{m} \\
A \otimes A\ar[r]^{m} & A }
\hspace{30pt}
\xymatrix{ } $$

A Hom-associative algebra $A$ is called unital if there exists a linear map $\eta: k\longrightarrow A$ where $\a \circ \eta=\eta$, and
$$m \circ (\Id \ot \eta)= m \circ ( \eta\ot\Id) =\a.$$ The unit element of $A$ is $\eta(1_k)=1_A$.
These conditions in terms of an element $a\in A$ can be written as $\a(1_A)=1$ and $a 1_A= 1_A a =\a(a)$.
The unitality condition in terms of a commutative diagram is

  $$
  \xymatrix{
A \ar[r]^{\eta\ot Id} \ar[rd]_{\a} &
A\ot A \ar[d]^{m} &
A \ar[l]_{\eta\ot Id} \ar[ld]^{\a}\\
&A}
$$

In many examples $\a$ is   an algebra map, i.e,    $\a(xy)= \a(x) \a(y)$ for all $x, y\in A$.
When $\a=\Id$, then we obtain the definition of  associative algebras.

 \begin{example}{\rm
Let $A$ be an algebra with multiplication  $m: A\ot A\longrightarrow A$, and $\a: A \longrightarrow A$ be an algebra map. We twist the multiplication of $A$ by $\a$ to obtain a new multiplication $m_{\a}(x,y)= m( \a(x), \a(y))$. Then $(A,  m_{\a}, \a)$ is a Hom-algebra.}
\end{example}

 \begin{example}\label{2d}
{\rm This example is a special case of the last example in \cite{ms2}.
In this example we define a $2$-dimensional Hom-algebra $A$ with a basis $B=\{ e_1, e_2\}$. We define the multiplication by

$$m(e_1, e_1)=e_1, ~~~~~m(e_1, e_2)=m(e_2, e_1)=e_2, ~~~~~m(e_2, e_2)=e_2.$$

 We set $\a(e_1)= 2e_1-e_2$ and $\a(e_2)=e_2$. This Hom-algebra is unital and commutative with the unit element  $\eta(1)=e_1$.

}
 \end{example}

 An element $x$ in an unital Hom-associative algebra $ (A, \a)$ is called Hom-invertible \cite{lmt},  if there exists an element $x^{-1}$  and a non-negative integer $k\in \mathbb{N}$  such that
$$\a^k(x x^{-1}) = \a^k(x^{-1} x)= 1.$$ The element $x^{-1}$ is called a Hom-inverse and the smallest $k$ is the invertibility index of $x$.
The Hom-inverse may not be unique if it exists. However the authors in \cite{lmt} showed that the unit element $1_A$ is Hom-invertible, the
product of any two Hom-invertible elements is Hom-invertible and every inverse of a Hom-invertible element
is Hom-invertible.

For two Hom-algebras  $(A,\mu ,\alpha )$ and $(A^{\prime },\mu ^{\prime },\alpha ^{\prime })$
a linear map $f:A\rightarrow A^{\prime }$
is called  a  Hom-algebra morphism if
$$f(xy)= f(x) f(y), ~~~~~\text{and} ~~~   f(\a(x))= \a'(f(x)), ~~~\forall x, y\in A.$$

Now we recall  the dual notion of a Hom-algebra which is called a Hom-coalgebra \cite{ms2}, \cite{ms3}. A Hom-coalgebra is a triple $(A, \Delta, \b)$, where $C$ is a $\mathbb{K}$-vector space, $\Delta: C\longrightarrow C\ot C$ is linear map, called comultiplication,  and $\b: C\longrightarrow C $ a linear map satisfying the Hom-coassociativity condition,

 $$(\Delta\ot \b ) \circ   \Delta  =  ( \b \ot \Delta) \circ   \Delta  .$$

If we use the Sweedler notation $ \Delta (c) = c^{(1)} \ot c^{(2)}$, then the coassociativity condition can be written as
$$ \b( c^{(1)}) \ot c^{(2)(1)} \ot c^{(2)(2)}  = c^{(1)(1)} \ot c^{(1)(2)}\ot \b(c^{(2)}).$$

The coassociativity property in terms of a commutative diagram is the dual of the one for the Hom-associativity of Hom-algebras as follows;
$$
\xymatrix{
C\ot C\ot C   & C\ot C\ar[l]_{\Delta\ot \b}  \\
C \otimes C \ar[u]^{\b\ot \Delta}& C\ar[l]_{\Delta}\ar[u]_{\Delta} }
\hspace{30pt}
\xymatrix{ } $$

A Hom-coassociative coalgebra is said to be counital if there exists a linear map $\varepsilon: C \longrightarrow \mathbb{K}$ where
$$(\Id\ot \varepsilon ) \circ   \Delta  =  ( \varepsilon \ot \Id) \circ   \Delta =\b .$$
This means $$c^{(1)} \varepsilon (c^{(2)})= \varepsilon (c^{(1)})c^{(2)} = \b(c).$$

Furthermore the map $\b$ is counital, i.e,  $\varepsilon (\b(c))= \varepsilon (c)$.
The counitality condition in terms of a commutative diagram is
 $$
  \xymatrix{
C \ar[r]^{\Delta} \ar[rd]_{\b} &
C\ot C \ar[d]^{\varepsilon \ot \Id}_{\varepsilon \ot \Id} &
C \ar[l]_{\Delta} \ar[ld]^{\b}\\
&C}
$$
Moreover if the map $\b$ is a coalgebra map then we have $ \Delta \circ \b= (\b \ot \b)\circ \Delta$.

\begin{example}\label{2dco}{\rm
This example is a special case of the last example in \cite{ms2}. The Hom-algebra introduced in Example \ref{2d} is a Hom-coalgebra by
\begin{align*}
  &\Delta(e_1)=e_1\ot e_1, ~~~~~~ \Delta(e_2)= e_1\ot e_2 + e_2\ot e_1 -2 e_2\ot e_2\\
  &\varepsilon(e_1)=1, ~~~~~~~~~~~~~~~ \varepsilon(e_2)=0.
\end{align*}
We set $\b(e_1)= e_1+e_2$ and $\b(e_2)= e_2$.
}

\end{example}

Let $(C, \Delta, \varepsilon, \b)$ and $(C', \Delta', \varepsilon', \b')$ be two Hom-coalgebras. A morphism $f: C\longrightarrow C'$ is called a Hom-coalgebra map if for all $x\in C$ we have $$ f(x)^{(1)}\ot f(x)^{(2)}=  f(x^{(1)})\ot f(x^{(2)}), ~~~~~ f\circ \b= \b' \circ f. $$

A $(\a, \b)$-Hom-bialgebra is a tuple $(B, m, \eta, \a,  \Delta, \varepsilon, \b)$ where $(B, m, \eta, \a)$ is a Hom-algebra and $(B, \varepsilon, \Delta, \b)$ is a Hom-coalgebra where $\Delta$ and $\varepsilon$ are morphisms of Hom-algebras, that is\\

i) $\Delta$ is a Hom-algebra map,  $\Delta (hk)= \Delta(h) \Delta(k)$, which is

$$(hk)^{(1)}\otimes (hk)^{(2)}= h^{(1)}k^{(1)}\otimes h^{(2)}k^{(2)}, \quad \forall~~ h,k\in B,$$

ii) $\Delta$ is unital; $\Delta(1)= 1\otimes 1.$

iii) $\varepsilon$ is a Hom-algebra map; $\varepsilon(xy)= \varepsilon(x) \varepsilon (y)$.

iv)  $\varepsilon$ is unital; $\varepsilon(1)=1$.

v) $\varepsilon (\a(x))= \varepsilon(x)$.

The algebra map property of $\Delta$  in terms of commutative diagrams is
\[ \underset{ { \text{{\bf{ }}}}} {
\xymatrixcolsep{5pc}\xymatrix{
B\ot B \ar[r]^-{\Delta m} \ar[d]_-{\Delta \ot \Delta} &B\ot B  \\
B\ot B\ot B\ot B  \ar[r]_-{id \ot \tau \ot id}   & B\ot B \ot B\ot B \ar[u]_-{m \ot m} }
 }
\] Here the linear map $\tau:B\ot B \ra B\ot B$ is given by $\tau(h\ot k)=k\ot h$.

The map $\epsilon$ being an algebra morphism  in terms of a commutative diagram   means

\[ \underset{ { \text{{\bf{  }}}}} {
\xymatrixcolsep{5pc}\xymatrix{
B \ot B \ar[r]^-{m}  \ar[d]_-{\epsilon \ot \epsilon}&B\ \ar[ld]^-{\epsilon}\\
k\ot k =k }}
 \]

\begin{remark}{\rm
  It can be proved that $\Delta$ and $\varepsilon$ are morphisms of unital Hom-algebras if and only if $m$ and $\eta$ are morphism of Hom-coalgebras.}
\end{remark}

\begin{example}\label{2dbi}{\rm

The $2$-dimensional Hom-algebra in Example \ref{2d} is a $(\a, \b)$-Hom-bialgebra by the coalgebra structure given in Example \ref{2dco}. See \cite{ms2}.

}

\end{example}

\begin{example}\label{twisted Hom-bialgebra}
  {\rm
  Let $ (B, m, \eta, \Delta, \varepsilon)$ be a bialgebra and $\a: B\longrightarrow B$ be a bialgebra map. Then $ (B, m_{\a}=\a\circ m, \a, \eta, \Delta_{\a}=\Delta\circ \a,  \varepsilon, \a)$ is a $(\a, \a)$-Hom-bialgebra.
  }
\end{example}

 Let $(B, m, \eta, \a,  \Delta, \varepsilon, \b)$ and $(B', m', \eta', \a',  \Delta', \varepsilon', \b')$ be two Hom-bialgebras. A morphism $f: B\longrightarrow B'$ is called a map of Hom-bialgebras of it is both morphisms of Hom-algebras and Hom-coalgebras.
  Let $(B, m, \eta, \a,  \Delta, \varepsilon, \b)$ be a Hom-algebra. The authors in \cite{ms2}, \cite{ms3}, showed that  $(\Hom(B, B), \star, \gamma)$ is an unital  Hom-algebra with $\star$ is  the convolution product
  $$f \star g = m \circ (f\ot g) \circ \Delta,$$
 and $ \gamma \in \Hom(B, B)$ is defined by $\gamma(f)= \a \circ f \circ \b.$ The unit is $\gamma \circ \varepsilon$.
Similarly if  $(A, m, \eta, \a)$  and $(C, \varepsilon, \Delta, \b)$   are   a Hom-algebra and a   Hom-coalgebra, respectively, then $(\Hom(C, A), \star, \gamma)$ is an unital  Hom-algebra where  $\star$ is  the convolution product.\\

Here we recall the original definition of Hom-Hopf algebras.

\begin{remark}\label{old hom-hopf}{\rm
  The notion of Hom-Hopf algebras first was appeared in   \cite{ms2} and \cite{ms3} as follows. A $(\a, \b)$-Hom-bialgebra  $(H, m, \eta, \a,  \Delta, \varepsilon, \b)$  with an antipode $S:H\longrightarrow H$ is called a $(\a, \b)$-Hom-Hopf algebra. A map $S$ is called antipode if it is an inverse  of the identity map  $\Id: H\longrightarrow H$
in  the Hom-associative algebra $\Hom(H,H) $ with respect to the multiplication given by the convolution product, i.e. $S \star \Id = \Id \star S= \eta \circ \varepsilon$. In fact for all $h\in H$ we have
$$S(h^{(1)}) h^{(2)}= h^{(1)} S(h^{(2)})= \varepsilon(h) 1.$$
This is the same as  usual definition of an antipode for Hopf algebras. The following properties of antipode of Hom-Hopf algebras with this definition were proved in \cite{cg} and \cite{ms2}. For all $x, y\in H$ we have;\\

i) If $\a=\b$ then $S\circ \a=\a\circ S$.

ii) The antipode $S$ of a Hom-Hopf algebra is unique.

iii) $S$ is anti-algebra map, i.e, $S(xy)= S(y) S(x)$.

iv) $S$ is anti-coalgebra map, i.e, $S(x)^{(1)}\ot S(x)^{(2)}= S(x^{(2)})\ot S(x^{(1)})$.

v) $S$ is unital, i.e, $S(1)=1$.

vi) $S$ is counital, i.e, $ \varepsilon(S(x))= \varepsilon(x)$.
}
\end{remark}

In this paper we use the recent  notion of Hom-Hopf algebras introduced in \cite{lmt}.
\begin{definition}\label{Hom-Hopf}\cite{lmt}
  Let $(B, m, \eta, \a,  \Delta, \varepsilon, \b)$ be a $(\a, \b)$-Hom-bialgebra. An anti-algebra,  anti-coalgebra morphism  $S: B\longrightarrow B$  is
said to be an antipode if\\

a) $S\circ \a= \a \circ S$.

b) $S \circ \eta = \eta$ and $\varepsilon \circ S= \varepsilon$.

c) $S$ is a relative Hom-inverse of the identity map $\Id : B \longrightarrow B$  for the convolution product, i.e, for any $x\in B$, there exists $k\in \mathbb{N}$ such that

\begin{equation}
  \a^k \circ (S\ot \Id)\circ\Delta(x) = \a^k \circ (\Id \ot S)\circ\Delta(x)= \eta \circ \varepsilon (x).
\end{equation}
A $(\a, \b)$-Hom-bialgebra with an antipode is called a $(\a, \b)$-Hom-Hopf algebra.
\end{definition}

One notes that  Definition \ref{Hom-Hopf}(c) in terms of Sweedler notation can be written as follows:

\begin{equation}
  \a^k(S(x^{(1)}) x^{(2)}) = \a^k (x^{(1)} S(x^{(2)}))= \varepsilon(x)1_B.
\end{equation}

\begin{remark}{\rm

There are some   differences between the old definition  of Hom-Hopf algebras in Remark \ref{old hom-hopf}  and the recent one in Definition \ref{Hom-Hopf}.
   The    Definition \ref{Hom-Hopf}(a)  in the special case of $\a=\b$ is followed by the definition of  Hom-Hopf algebras in Remark \ref{old hom-hopf}(i).
  Also   Definition \ref{Hom-Hopf}(b) is the result of  the old definition in Remark \ref{old hom-hopf}(v)(vi).
Furthermore  the antipodes of Hom-Hopf algebras in Definition \ref{Hom-Hopf} are the relative Hom-inverse of the  identity map whereas the antipode in Remark \ref{old hom-hopf} is actually the inverse of the identity map. Finally the antipode in Remark \ref{old hom-hopf}  is unique however the antipode in Definition
 \ref{Hom-Hopf} is not necessarily unique. In fact the authors in \cite{lmt} proved that if $S$ and $S'$ are two antipodes for the Hom-Hopf algebra $H$ in the sense of Definition \ref{Hom-Hopf}, then for every $x\in H$ there exist $k\in \mathbb{N}$ where
 $$\a^{k+2} \circ S \circ \b^2(x) = \a^{k+2} \circ S' \circ \b^2(x).$$
 In special case when $\a$ and $\b$ are both invertible then $S=S'$ and the antipode is unique.}
\end{remark}

\begin{proposition}
  Any Hom-Hopf algebra $(H, m, \eta, \a,  \Delta, \varepsilon, \b, S)$  in the sense of Remark \ref{old hom-hopf} which satisfies  the extra condition $S\circ \a= \a \circ S$, is a Hom-Hopf algebra in the sense of Definition \ref{Hom-Hopf}  where  $k=1$ for all elements $x\in H$.
\end{proposition}
\begin{proof}
  By Remark \ref{old hom-hopf}  the antipode $S$ is an unital,  counital,   anti-algebra,  and anti-coalgebra map. Therefore for $k=1$ it satisfies all the conditions of Hom-Hopf algebras in Definition \ref{Hom-Hopf}.
\end{proof}

\begin{corollary}
  If for Hom-Hopf algebra $(H, m, \eta, \a,  \Delta, \varepsilon, \b, S)$ in the sense of Remark \ref{old hom-hopf} satisfies  $\a=\b$ then $H$ is a Hom-Hopf algebra in the sense of Definition \ref{Hom-Hopf}.
\end{corollary}
\begin{proof}
  If $\a=\b$ then $\a$ is a map of Hom-bialgebras and  by Remark \ref{old hom-hopf} we have  $\a \circ S= S\circ \a$.  Now  the result is followed by  the previous Proposition.
\end{proof}
A Hom-Hopf algebra is called commutative if it is commutative as Hom-algebra and it is called cocommutative if it cocommutative as Hom-coalgebra.

\begin{example}\label{HH}
{\rm
Let $(H, m, \eta, \a,  \Delta, \varepsilon, \b, S)$  and Let $(K, m', \eta', \a',  \Delta', \varepsilon', \b', S')$ be two Hom-Hopf algebras. Then
$H\ot K$ is also a Hom-Hopf algebra by multiplication $m\ot m$, unit $\eta\ot \eta'$, and $\a\ot \a': H\ot H \longrightarrow H\ot H$, the coproduct $\Delta\ot \Delta'$ and counit $\varepsilon \ot \varepsilon'$ and the linear map $\b\ot \b': H\ot H \longrightarrow H\ot H$.
  }
\end{example}

\begin{definition}
 Let  $(H, m, \eta, \a,  \Delta, \varepsilon, \b, S)$ be a $(\a, \b)$-Hom-Hopf algebra. An element $h\in H$ is called a group-like element if $\Delta(h)=h\ot h$ and $\b(h)=h$.
\end{definition}

\begin{remark}{\rm
  If $h\in H$ is a group-like element then $ \varepsilon(h)h=\b(h)=h$. Therefore $\varepsilon(h)=1$.}
\end{remark}
One notes that the authors in \cite{lmt} introduced group-like elements with condition $\varepsilon(h)=1$ which in fact implies $\b(h)=h$. Therefore their definition is equivalent to the one in this paper. However we preferred to have $\b(h)=h$ as definition and similar as ordinary Hopf algebras the condition $\varepsilon(h)=1$ is the  result of the fact that $h$ is a group-like element. The notion of Hom-groups  introduced in \cite{lmt} and studied in \cite{hassan}.
 For any Hom-group  $(G, \a)$, the author in \cite{hassan} introduce  the Hom-group algebra $\mathbb{K}G$.

\begin{proposition}
  For any Hom-group $(G, \a) $, the Hom-group algebra $\mathbb{K}G$ is a $(\a, \Id)$-Hom-Hopf algebra.
\end{proposition}
\begin{proof}

We define the coproduct by    $\Delta (g)= g\ot g$, counit by $\varepsilon(g)=1$, and the  antipode by $S(g)=g^{-1}$.
  Since $\b=\Id$ one  verifies that $\mathbb{K}G$ is a  $(\a, \Id)$-Hom-bialgebra.
    Since for any Hom-group we have  $\a(g)^{-1}= \a(g^{-1})$ then $ S(\a(g))=\a(S(g))$. The unit element of $\mathbb{K}G$ is $1_G$ and therefore  $S(1_G)=1_G^{-1}= 1_G$. Furthermore $\varepsilon (S(g))= \varepsilon (g^{-1})=1$. Finally if the invertibility index of $g\in G$ is $k$ then
    $$\a^k( S(g) g)= \a^k(g^{-1} g)=1.$$
\end{proof}

  One notes that all elements $g\in \mathbb{K}G$ are group-like elements. Also $\mathbb{K}G$ is a cocommutative Hom-Hopf algebra. If $G$ is an abelian Hom-group then $\mathbb{K}G$ is a commutative Hom-Hopf algebra.
The authors in \cite{lmt}, proved that  set of group-like elements of a Hom-Hopf algebra is a Hom-group.
In the Hom-bialgebra structure of $\mathbb{K}G$ one can define the comultiplication by $\Delta(g) = \a(g)\ot \a(g)$ to obtain a $(\a, \a)$-Hom-bialgebra.

\begin{example}\label{polynomial}
  {\rm (Hom-bialgebra of quantum matrices)

  In this example we study a $4$-dimensional Hom-bialgebra which is not a Hom-Hopf algebra. First we recall the construction of quantum matrices from \cite{es}, \cite{k}, \cite{m}, \cite{s}.
  Let $q\in \mathbb{K}$ where $q\neq 0$ and $q^2\neq -1$.
  Let $\mathcal{O}_q(M_2)= \mathbb{K}[ a, b, c, d]$ be the  polynomial algebra with variables $a,b,c,d$  satisfying the following relations

  \begin{align*}
    &ab = q^{-1} ba, ~~~~~~ bd= q^{-1}db, ~~~~ac = q^{-1}ca, ~~~~ cd=q^{-1}dc\\
    &bc=cb, ~~~~~ ad-da= (q^{-1}-q)bc.
  \end{align*}

   Clearly $\mathcal{O}_q(M_2)$ is not commutative except $q=1$.

   We define a coproduct as follows.

\begin{align*}
  &\Delta(a) = a\ot a + b\ot c, ~~~~~~~ \Delta(b)= a\ot b+ b\ot d\\
 & \Delta (c) = c\ot a+  d\ot c, ~~~~~~~ \Delta(d) = c\ot b+ d\ot d
\end{align*}
If we consider the elements of $\mathcal{O}(M_2)$ as $2 \times 2$ matrices with entries in $\mathbb{K}$ then \\

 $~~~~~~~~~~~~~~~~~~~~~~~~~~~~~~~~\begin{bmatrix}
  \Delta(a)& \Delta(b)\\ \Delta(c) &\Delta(d)
\end{bmatrix}_{\mathcal{O}(M_2)}=  \begin{bmatrix}
  a&b\\c&d
\end{bmatrix} \ot  \begin{bmatrix}
  a&b\\c&d
\end{bmatrix} $\\
This comultiplication is not cocommutative.
We define the counit by
\begin{align*}
  \varepsilon(a)=\varepsilon(d)=1, ~~~~~~~~~~~~~~~~~~~~ \varepsilon(b)=\varepsilon(c)=0.
\end{align*}

This coproduct and counit defines a bialgebra structure on $\mathcal{O}(M_2)$. Now we explain the Hom-bialgebra structure from  \cite{ya1}. We define a bialgebra map $\a : \mathcal{O}(M_2)\longrightarrow \mathcal{O}(M_2)$ by
\begin{align*}
  \a(a)=a, ~~~ \a(b)=\lambda b, ~~~ \a(c)=\lambda^{-1}c, ~~~ \a(d)=d.
\end{align*}
where $\lambda \in \mathbb{K}$ is any invertible element. In fact\\

$$
\a(\begin{bmatrix}
  a&b\\c&d
\end{bmatrix}) = \begin{bmatrix}
  \a(a)&\a(b)\\\a(c)&\a(d)
\end{bmatrix}= \begin{bmatrix}
  a&\lambda b\\\lambda^{-1}c&d
\end{bmatrix}
$$\\

It can be verified that $\a$ is a bialgebra morphism. One notes that $\varepsilon\circ \a=\varepsilon$. Now we use $\a$ to twist both product and coproduct of $\mathcal{O}(M_2)$ as explained in Example \ref{twisted Hom-bialgebra} to obtain a $(\a, \a)$- Hom-bialgebra  $\mathcal{O}_q(M_2)_{\a}$. Therefore the coproduct  of  $\mathcal{O}_q(M_2)_{\a}$     is

\begin{align*}
 &\Delta(a) = a\ot a + b\ot c, ~~~~~~~~~~~~~~ \Delta(b)= \lambda a\ot b+ \lambda b\ot d\\
 & \Delta (c) = \lambda^{-1}c\ot a+ \lambda^{-1} d\ot c, ~~~~~~~ \Delta(d) = c\ot b+ d\ot d
\end{align*}

In fact

$~~~~~~~~~~~~~~~~~~~~~~~~~~~~~~~~\begin{bmatrix}
  \Delta(a)& \Delta(b)\\ \Delta(c) &\Delta(d)
\end{bmatrix}_{\mathcal{O}_q(M_2)_{\a}}=  \begin{bmatrix}
  a&\lambda b\\\lambda^{-1}c&d
\end{bmatrix} \ot  \begin{bmatrix}
  a&\lambda b\\\lambda^{-1}c&d
\end{bmatrix} $\\

 Now we consider quantum determinant element
 $$det_q= \mu_{\a}(a, d) -q^{-1}\mu_{\a}(b, c )\in \mathcal{O}_q(M_2)_{\a}. $$
 One notes that $$det_q= \a(a)\a(d)- q^{-1} \a(b)\a(c)= ad - q^{-1} (\lambda b) ( \lambda^{-1} c)= ad - q^{-1} bc.$$
Similarly $\a( det_q)= det_q$. Therefore
$$\Delta_{\mathcal{O}_q(M_2)_{\a}}(det_q) = \Delta \a(det_q) = \Delta (det_q).$$

It is shown in  \cite{k} and  \cite{s} that $\Delta(det_q) = det_q \ot det_q$ which means $det_q$ is a group-like element.
Also  $\varepsilon_{ \mathcal{O}(M_2)_{\a}}= \varepsilon_{\mathcal{O}(M_2)}.$
Therefore  $\varepsilon(ad-q^{-1}bc)= 1$. Then $det_q$ is a group-like element of  Hom-bialgebra  $\mathcal{O}_q(M_2)_{\a}$. Since the set of group-like elements of a Hom-Hopf algebras is a Hom-group \cite{lmt}, then every group-like element is relative Hom-invertible. However $det_q$ is not clearly relative Hom-invertible by definition of $\a$. Therefore $\mathcal{O}_q(M_2)$ is not a Hom-Hopf algebra.
  }
\end{example}

\begin{example}\label{2-dimensional Hopf}{\rm
  The $2$-dimensional bialgebra $H$, in Example \ref{2dbi}, is a $(\a, \b)$-Hom-Hopf algebra by
  \begin{align*}
    S(e_1)=e_1, ~~~~~~~~   S(e_2)=e_2.
  \end{align*}
  It is straightforward to check that $S(h^{(1)}) h^{(2)}=  h^{(1)} S(h^{(2)}) =  \varepsilon(h) 1$ for all $h\in H$.
  Therefore it is a Hom-Hopf algebra in the sense of Remark \ref{old hom-hopf}. Since $S=\Id$ then $S\circ \a= \a \circ S$ and  $H$ is also Hom-Hopf algebra in the sense of    Definition \ref{Hom-Hopf}.
  }
\end{example}

\section{Properties of antipodes}

In this section we study the properties of antipods for Hom-Hopf algebras. We remind  that we are using Definition \ref{Hom-Hopf}.
We need the following basic properties of convolution product for later results.

\begin{remark}\label{convolution}
 Let $(H, m, \eta, \a,  \Delta, \varepsilon, \b)$    be a $(\a, \b)$-Hom-bialgebra.
 We consider the convolution Hom-algebra $\Hom( H, H)$. If $f, g\in \Hom(H, H)$, then the authors in \cite{lmt}, showed that \\

 i) $ \a^n ( f\star g) = \a^n f \star \a^n g$.

 ii) $f \star (\eta \circ\varepsilon )= \a \circ f \circ \b= (\eta \circ\varepsilon ) \star f$.
\end{remark}
The following Proposition will play an important rule for the further results in this paper.

\begin{proposition}\label{important}
  Let  $(H, m, \eta, \a,  \Delta, \varepsilon, \b, S)$ and $(A, m', \eta', \a',  \Delta', \varepsilon', \b', S')$  be $(\a, \b)$ and $(\a', \b')$-Hom-bialgebras and $\Hom(H, A)$ be the convolution Hom-algebra.
  If $f, g, \varphi \in Hom(H, A)$  where $f$ and $g$ are the relative Hom-inverse of $\varphi$, then for every $x\in H$ there exists $k\in\mathbb{ N}$ where
  $$\a'^{k+2} \circ f\circ \b^2 (x)= \a'^{k+2} \circ g \circ \b^2(x).$$
\end{proposition}
\begin{proof}
  For $x\in H$ there exist $k', k''\in \mathbb{N}$ where
  $$ \a'^{k'} ( f\star \varphi)(x) = \a'^{k'} (\varphi \star f)(x)= \varepsilon(x)1 = \eta \circ \varepsilon(x) ,$$ and
  $$ \a'^{k''} ( g\star \varphi)(x) = \a'^{k''} (\varphi \star g)(x)=\varepsilon(x)1 =\eta \circ \varepsilon(x). $$
  Let $k= \max (k', k'')$. We ignore  the composition sign for easier computation. We have
  \begin{align*}
    \a'^{k+2} f \b^2 &= \a' (\a'^{k+1} f\b) \b\\
    &=(\a'^{k+1} f\b) \star (\eta \varepsilon)= (\a'^{k+1} f\b)\star \a'^k(\varphi \star g)\\
    &=(\a'^{k+1} f\b)\star (\a'^k\varphi \star \a'^kg)= (\a'^{k} f\b)\star \a'^k \varphi) \star \a'^{k+1}g\\
    &=(\a'^k f \star \a'^k \varphi)\star \a'^{k+1}g\b= \a'^k ( f\star \varphi) \star \a^{k+1} g\b\\
    &=\a'^k(\eta \varepsilon) \star \a^{k+1} g\b=\eta \varepsilon \star \a^{k+1} g\b\\
    &= \a^{k+2}g\b^2.
  \end{align*}
  We used the Remark \ref{convolution}(ii) in the second equality, the Remark \ref{convolution}(i) in the fourth equality, the Hom-associativity in the fifth equality, the Hom-coassociativity in the sixth equality, the Remark \ref{convolution}(i) in the seventh equality, unitality of $\a'$ in the eight equality, and  the Remark \ref{convolution}(ii) in the last equality.
\end{proof}

The previous proposition shows that the relative Hom-inverses are unique in some sense. In fact if $\a$ and $\b$ are invertible then $f=g$.

\begin{proposition}\label{relative anti algebra}
Let  $(H, m, \eta, \a,  \Delta, \varepsilon, \b)$ be a Hom-bialgebra  with  multiplicative $\a$,   endowed with a map $S: H\longrightarrow H$ where $S$ is a relative Hom-inverse of the identity map $\Id : H \longrightarrow H$ in $\Hom(H, H)$.
  Let $P(x, y)=S(x y)$, $N(x,y)= S(y) S(x)$, and $ M(x, y)= xy$.
  Then $S(xy)$ and $S(y)S(x)$ are both relative Hom-inverse of the multiplication $M(x, y)=xy$ in $\Hom(H\ot H, H)$, with respect to the convolution product, and for every $x, y\in H$ there exists $K\in \mathbb{N}$ where

 \begin{equation}\label{relative hom-anti algebra}
   \a^{K+2}( S[\b^2(x) \b^2(y)])=  \a^{K+2} [S(\b^2(y)) S(\b^2(x))].
 \end{equation}

\end{proposition}
\begin{proof}
For any $x, y\in H$ there exists $k', k''\in \mathbb{N}$, such that
$$ \a^{k'}(S(x^{(1)}) x^{(2)}) = \a^{k'} (x^{(1)} S(x^{(2)}))= \varepsilon(x)1_B,    $$
and
$$ \a^{k''}(S(y^{(1)}) y^{(2)}) = \a^{k''} (y^{(1)} S(y^{(2)}))= \varepsilon(y)1_B.    $$
Let $k=\max(k', k'')+2$. Then
\begin{align*}
  \a^k(M\star N)( x,y) = &\a^k[M(x^{(1)}, y^{(1)}) N(x^{(2)},y^{(2)})]\\
  &=\a^k([x^{(1)} y^{(1)}][S(y^{(2)}) S(x^{(2)}])\\
  &=[\a^k(x^{(1)})\a^k(y^{(1)})] \a^k ( S(y^{(2)}) S(x^{(2)}))\\
  &=\a^{k+1} (x^{(1)}) [\a^k (y^{(1)})  \a^{k-1} ( S(y^{(2)}) S(x^{(2)}))  ]\\
  &=\a^{k+1} (x^{(1)}) [\a^k (y^{(1)}) [ \a^{k-1} ( S(y^{(2)}))\a^{k-1}( S(x^{(2)}))]  ]\\
   &=\a^{k+1} (x^{(1)}) [[\a^{k -1}(y^{(1)}) \a^{k-1}( S(y^{(2)})]\a^{k}( S(x^{(2)}) ]\\
    &=\a^{k+1} (x^{(1)}) [\a^{k -1}(y^{(1)} S(y^{(2)}))\a^{k}( S(x^{(2)}) ]\\
     &=\a^{k+1} (x^{(1)})[\varepsilon(y) 1\a^{k}( S(x^{(2)})]\\
      &=\a^{k+1} (x^{(1)})\a^{k+1}( S(x^{(2)}) \varepsilon(y)\\
      &=\a^{k+1} [x^{(1)} S(x^{(2)})] \varepsilon(y)\\
      &=\a^{k+1}( \varepsilon (x) 1) \varepsilon(y)\\
      &= \varepsilon(x) \varepsilon(y)1.
\end{align*}
We used the Hom-associativity property in the fourth equality, the Hom-unitality in ninth equality. Therefore $ N(x,y)= S(y) S(x)$  is a relative Hom-inverse of
$ M(x,y)= xy$. Now for $xy\in H$ there exists $n\in \mathbb{N}$ where
$$ \a^{n}(S(x^{(1)}) x^{(2)}) = \a^{n} (x^{(1)} S(x^{(2)}))= \varepsilon(x)1_B.    $$
Then
\begin{align*}
  &\a^n( P\star  M)(x, y)= \a^n(P(x^{(1)}, y^{(1)}) M(x^{(2)},y^{(2)}))\\
  &= \a^n [S(x^{(1)} y^{(1)}) x^{(2)}y^{(2)}]= a^n [S ((xy)^{(1)}) (xy)^{(2)}]= \varepsilon (xy) 1.
\end{align*}
Therefore $ P(x, y) = S(xy)$ is a relative Hom-inverse of     $ M(x,y)=xy$.
Then $S(xy)$ and $S(y)S(x)$ are both relative Hom-inverse of the multiplication $M(x, y)=xy$ in $\Hom(H\ot H, H)$ with respect to the convolution product. Therefore by Proposition \ref{important} there exists $K\in \mathbb{N}$ such that

$$ \a^{K+2} \circ P \circ \b_{H\ot H}^2 (x,y)=  \a^{K+2} \circ N \circ \b_{H\ot H}^2(x,y) .$$
By Example \ref{HH}, we have $\b_{H\ot H}= \b_H\ot \b_H$ and therefore we obtain the result.

\end{proof}
The relation \ref{relative hom-anti algebra} is called the relative Hom-anti algebra map property of $S$.
\begin{corollary}
   Let $(H, m, \eta, \a,  \Delta, \varepsilon, \b)$ be a $(\a, \b)$-Hom-bialgebra, where $\a$ is multiplicative  and  $\a$ and $\b$ are invertible. If $H$ is endowed with a linear map  $S: H \longrightarrow H$  where  $S$ is a relative Hom-inverse of the identity map $\Id : H \longrightarrow H$  in $\Hom(H, H)$,
 then $S$ is an anti-algebra map.
\end{corollary}
\begin{proof}
  By  previous Proposition we have $ \a^{K+2} \circ P \circ \b_{H\ot H}^2 (x,y)=  \a^{K+2} \circ N \circ \b_{H\ot H}^2(x,y) .$ Since $\a$ and $\b$ are invertible then $P= N$ or  $S(xy)=S(y) S(x)$.
\end{proof}

Similarly we have the following proposition.

\begin{proposition}\label{relative anti coalgebra}
Let  $(H, m, \eta, \a,  \Delta, \varepsilon, \b)$ be a  Hom-bialgebra where   $\b$ is a coalgebra map. Assume $H$ is   endowed with a map $S: H\longrightarrow H$ where $S$ is a relative Hom-inverse of the identity map $\Id : H\longrightarrow H$  with respect to  the convolution product.
  Let $P(x)=\Delta(S(x))$, $N= \tau(S\ot S)\Delta$, and $ \Delta(x)= x^{(1)} \ot x^{(2)}$ where $\tau(x,y)= (y, x)$ for all $x, y\in H$.
  Then $P$ and $N$ are both Hom-relative inverse of the comultiplication $\Delta$ in $\Hom(H, H\ot H)$ with respect to the convolution product and  for every $x\in H$ there exists $k\in \mathbb{N}$ such that

   \begin{equation}\label{relative hom-anti coalgebra}
     \a^{k+2}[ S(\b^2(x) )^{(1)}] \ot \a^{k+2}[ S(\b^2(x) )^{(2)}] =  \a^{k+2} [S(\b^2(x^{(2)}))]\ot \a^{k+2} [S(\b^2(x^{(1)}))].
   \end{equation}

\end{proposition}

\begin{proof}
Using proposition \ref{important} and Example \ref{HH},  the proof is similar to the previous Proposition.
\end{proof}
The relation \ref{relative hom-anti coalgebra} is called the relative Hom-anti coalgebra map property of $S$.
In special case we have the following.
\begin{corollary}
   Let $(H, m, \eta, \a,  \Delta, \varepsilon, \b)$ be a $(\a, \b)$-Hom-bialgebra,  where $\b$ is coalgebra morphism and  $\a$ and $\b$ are invertible. Assume $H$ is endowed with a linear map  $S: H \longrightarrow H$  where  $S$ is a relative Hom-inverse of the identity map $\Id : H \longrightarrow H$  with respect to  the convolution product.
 Then $S$ is an anti-coalgebra map.
\end{corollary}

\begin{proposition}\label{hom unitality prop}
  Let  $(H, m, \eta, \a,  \Delta, \varepsilon, \b)$ be a  Hom-bialgebra ,   endowed with a map $S: H\longrightarrow H$ where $S$ is a relative Hom-inverse of the identity map $\Id : H \longrightarrow H$  with respect to the the convolution product. Then there exists $k\in \mathbb{N}$ such that
\begin{equation}\label{hom-unitality}
  \a^{k+1}( S(1))=1.
\end{equation}

\end{proposition}
\begin{proof}
We apply relative Hom-invertibility of $S$ for $h=1$. So there exist $k\in \mathbb{N}$ where
$$
   1= \varepsilon(1) 1= \a^k( \Id \star S)(1)= \a^k(1S(1)) = \a^{k+1}(S(1)).
 $$
\end{proof}

The condition  \ref{hom-unitality} is called the relative Hom-unitality property of  $S$.

\begin{proposition}\label{relative counitality}

 Let  $(H, m, \eta, \a,  \Delta, \varepsilon, \b)$ be a  Hom-bialgebra ,   endowed with a map $S: H\longrightarrow H$ where $S$ is a relative Hom-inverse of the identity map $\Id : H \longrightarrow H$  with respect to the the convolution product. Then there exists $k\in \mathbb{N}$ such that

\begin{equation}\label{hom-counitality}
  \varepsilon (\a^k(S(h)))=1 \varepsilon(h).
\end{equation}
\end{proposition}
\begin{proof}
  For any $h\in H$ there exists $k\in \mathbb{N}$ such that

  $$\varepsilon(h) 1= \a^k( S(h^{(1)}h^{(2)})).$$
  Therefore $$\varepsilon (\varepsilon(h)1) = \varepsilon(\a^k( S(h^{(1)}h^{(2)}))).$$
  Since $\varepsilon$ is unital and it commutes with $\a$ then

  $$1\varepsilon(h) = \a^k ( \varepsilon(S(h^{(1)}h^{(2)}))    )  = \a^k (\varepsilon(S(h^{(1)}) \varepsilon(h^{(2)}) ).  $$
  Therefore
  $$1\varepsilon(h) = \a^k( \varepsilon(S(h^{(1)}\varepsilon(h^{(2)}))    ) = \a^k \varepsilon(S(h))= \varepsilon (\a^k(S(h))).$$
\end{proof}
The condition  \ref{hom-counitality} is called the relative Hom-counitality property of  $S$.

\begin{lemma}
  Let  $(H, m, \eta, \a,  \Delta, \varepsilon, \b, S )$ be a Hom-Hopf algebra,   $(A, m', \eta', \a'  )$ be a Hom-algebra and $f: H\longrightarrow A$ be a Hom-algebra map. Then $f \circ S$ is a  relative Hom-inverse of $f$ in $\Hom (H, A)$.

  \end{lemma}
  \begin{proof}
    We show that $f\circ S$ is a  relative Hom-inverse of $f$ in $ \Hom( H, A)$.
    For any  $h\in H$  there exist $k\in \mathbb{N}$ where  $\a^k(S(h^{(1)}) h^{(2)}) = \a^k (h^{(1)} S(h^{(2)}))= \varepsilon(x)1.$ Therefore
    \begin{align*}
      &\a^k( (f\circ S)\star f)(h) = \a^k (f( S(h^{(1)})) f(h^{(2)})) = \a^k(f(S(h^{(1)})h^{(2)})  )\\
      & = f (\a^k(S(h^{(1)})h^{(2)}))= f(\varepsilon(h)1)=\varepsilon(h)1.
    \end{align*}
    Similarly since $ \a^k (h^{(1)} S(h^{(2)}))= \varepsilon(x)1$ we have $\a^k(  f \star(f\circ S)  )(h)=\varepsilon(h)1$.
  \end{proof}
\begin{lemma}
  Let  $(H, m, \eta, \a,  \Delta, \varepsilon, \b, S )$ be a Hom-Hopf algebra,   $(C, \Delta', \varepsilon', \b'  )$ be a Hom-coalgebra and $f: C\longrightarrow H$ be a Hom-coalgebra map. Then  $ S\circ f$ is a  relative Hom-inverse of $f$ in $\Hom (C, H)$.

  \end{lemma}
  \begin{proof}
    Similar to the previous Lemma.
  \end{proof}

\begin{proposition}\label{Hopf map}
Let  $(H, m, \eta, \a,  \Delta, \varepsilon, \b, S)$ and $(K, m', \eta', \a',  \Delta', \varepsilon', \b', S')$  be $(\a, \b)$ and $(\a', \b')$-Hom-Hopf algebras.
If $f: H \longrightarrow K$ is a map of Hom-bialgebras then there exists $K\in \mathbb{N}$ such that

 \begin{align}\label{Hopf map condition}
   \a^K \circ ( f\circ S)\circ \b^2(h) = \a^K\circ ( S'\circ f)\b^2(h).
 \end{align}
\end{proposition}
\begin{proof}

 By the previous Lemmas  $f \circ S$  and  $ S' \circ f $ are  the relative Hom-inverse of $f$ in $ \Hom( H, K)$. Then the  result is followed by Proposition \ref{important}.

\end{proof}
The condition \ref{Hopf map condition} is called the relative Hom-Hopf algebra map property of $S$.
As a special case of the previous Proposition we have the following result.

\begin{corollary}\label{hopf morphism}
 Let  $(H, m, \eta, \a,  \Delta, \varepsilon, \b, S)$ and $(K, m', \eta', \a',  \Delta', \varepsilon', \b', S')$  be Hom-Hopf algebras where $\a, \b, \a',\b'$ are invertible. Then any Hom-bialgebra map $f: H \longrightarrow K$ is a Hom-Hopf algebra map, i.e,
  $$f\circ S= S' \circ f.$$
\end{corollary}

\begin{corollary}
   Let  $(H, m, \eta, \a,  \Delta, \varepsilon, \b, S)$ be Hom-Hopf algebra where $\a=\b$. Then there exists $k\in\mathbb{ N}$ such that

   \begin{equation}
     \a^k (\a\circ S) \a^2 (h) = \a^k( S\circ \a)\a^2 .
   \end{equation}
   If $\a$ is invertible then $ \a \circ S = S\circ \a$.
\end{corollary}
\begin{proof}
  Since $\a=\b$ then $\a$ is a map of Hom-bialgebras and the result is followed by Proposition \ref{Hopf map}.
\end{proof}

Here we summarize some of the results in this section.

\begin{theorem}
  Let  $(H, m, \eta, \a,  \Delta, \varepsilon, \b)$ be a Hom-bialgebra  where    $\a$  and $\b$ are  morphisms of algebra and   coalgebra, respectively. Assume $H$ is    endowed with a map $S: H\longrightarrow H$ where $S$ is a relative Hom-inverse of the identity map $\Id : H \longrightarrow H$  with respect to the the convolution product. Then $S$ is a relative Hom-anti-algebra map, a relative Hom-anti-coalgebra map,  relative Hom-unital, and relative Hom-counital, i.e, there exists $k\in\mathbb{ N}$ such that\\

 i) $\a^{k+2}( S(\b^2(x) \b^2(y))=  \a^{k+2} (S(\b^2(y)) S(\b^2(x))).$

 ii)$\a^{k+2}[ S(\b^2(x) )^{(1)}] \ot \a^{k+2}[ S(\b^2(x) )^{(2)}] =  \a^{k+2} [S(\b^2(x^{(2)}))]\ot \a^{k+2} [S(\b^2(x^{(1)}))].$

 iii) $\a^{k+1}( S(1))=1.$

 iv) $ \varepsilon (\a^k(S(h)))=1 \varepsilon(h)$.

 v) If $\a=\b$ then $\a^k (\a\circ S) \a^2 (h) = \a^k( S\circ \a)\a^2 .$\\

  Furthermore if $\a$ and $\b$ are invertible then $S$ is morphisms of algebras and coalgebras, respectively, and $\a \circ S= S \circ \a$.
\end{theorem}

\begin{proposition}
Let  $(H, m, \eta, \a,  \Delta, \varepsilon, \b, S)$ be a commutative Hom-Hopf algebra.  Then for every $x\in H$ there exist $k\in \mathbb{N}$ where
$$\a^{k+2} \circ S^2\circ \b^2 (x)= \a^{k+2} \circ \Id  \circ \b^2(x).$$
\end{proposition}
\begin{proof}
We show that $S^2$ is a relative Hom-inverse of $S$. For any  $h\in H$  there exist $k\in \mathbb{N}$ where  $\a^k(S(h^{(1)}) h^{(2)}) = \a^k (h^{(1)} S(h^{(2)}))= \varepsilon(x)1.$ Therefore
\begin{align*}
   \a^k(S\star S^2)(h)= & \a^k [S(h^{(1)}) S^2(h^{(2)})] \\
   &= \a^k [S[S(h^{(2)}) h^{(1)} ]]  = \a^k [S[ h^{(1)}S(h^{(2)}) ]]\\
   &=  S  [\a^k[ h^{(1)}S(h^{(2)})] ]= S(\varepsilon(h) 1)= \varepsilon(h)1.
\end{align*}
We used the anti-algebra map property of $S$ in the second equality, commutativity of $H$ in the third equality, commutativity of $S$ and $\a$ in the fourth equality, and the unitality of $S$ in the last equality. Similarly it can be shown that $\a^k (S^2 \star S)(h)= \varepsilon(h)1$. Therefore  $S^2$  and the identity map $\Id_H$ are both relative Hom-inverse of  $S$. Now the result is followed by  Proposition \ref{important}.
\end{proof}
As a special case of previous Proposition, we have the following result.
\begin{corollary}
   If  $(H, m, \eta, \a,  \Delta, \varepsilon, \b, S)$ is a commutative Hom-Hopf algebra with invertible $\a$ and $\b$ then

   $$S^2=\Id.$$
\end{corollary}

\begin{proposition}
Let  $(H, m, \eta, \a,  \Delta, \varepsilon, \b, S)$ be a cocommutative Hom-Hopf algebra.  Then for every $x\in H$ there exist $k\in \mathbb{N}$ where
$$\a^{k+2} \circ S^2\circ \b^2 (x)= \a^{k+2} \circ \Id  \circ \b^2(x).$$
\end{proposition}
\begin{proof}
We show that $S^2$ is a relative Hom-inverse of $S$. For any  $h\in H$  there exist $k', k''\in \mathbb{N}$ where  $\a^{k'}(S(h^{(1)}) h^{(2)}) = \a^{k'} (h^{(1)} S(h^{(2)}))= \varepsilon(x)1,$ and $\a^{k''}(S(S(h)^{(1)}) h^{(2)}) = \a^{k''} (h^{(1)} S(S(h)^{(2)}))= \varepsilon(x)1.$ Let $k=\max (k', k'')$.

Therefore
\begin{align*}
   \a^k(S\star S^2)(h)= & \a^k [S(h^{(1)}) S^2(h^{(2)})] \\
   &= \a^k [S(h)^{(2)} S (S(h)^{(1)} )]  = \a^k [S(h)^{(1)} S ( S(h)^{(2)}) ]\\
   &= \a^k(\varepsilon(S(h))1)= \varepsilon(S(h))\a^k(1)= \varepsilon(h) 1.
\end{align*}
We used the anti-coalgebra map property of $S$ in the second equality, cocommutativity of $H$ in the third equality, and the counitality of antipode in fifth equality. Similarly $\a^k (S^2 \star S))(h)= \varepsilon(h) 1$. Therefore  $S^2$ and the identity map $\Id_H$ are both Hom-relative inverse of  $S$. Then the result is followed by  Proposition \ref{important}.
\end{proof}
As a special case of the previous Proposition we have the following result.

\begin{corollary}
   If  $(H, m, \eta, \a,  \Delta, \varepsilon, \b, S)$ is a cocommutative Hom-Hopf algebra with invertible $\a$ and $\b$ then

   $$S^2=\Id.$$
\end{corollary}
\begin{proposition}
  Let  $(H, m, \eta, \a,  \Delta, \varepsilon, \b, S)$ be a Hom-Hopf algebra and $h\in H$ a primitive element, i.e, $\Delta(h)= 1\ot h+ h\ot 1$. Then there exists $k\in \mathbb{N}$ where
  \begin{equation}
    \a^{k+1} (S(h))= -\a^{k+1} (h).
  \end{equation}
\end{proposition}
\begin{proof} There exists $k\in \mathbb{N}$ such that  $\a^{k}(S(h^{(1)}) h^{(2)}) = \a^{k'} (h^{(1)} S(h^{(2)}))= \varepsilon(x)1.$ Therefore

$$\a^k( hS(1)) + \a^k(1S(h))=\varepsilon(h)1.$$ Since $S$ is unital and for any $x\in H$, we have $1x=\a(x)$, then
$$\a^{k+1}(h) + \a^{k+1}(h) = \epsilon(h)1.$$

By \cite{lmt},   for any primitive element $h$, we have $\varepsilon(h)=0$. Therefore $\a^{k+1} (S(h))= -\a^{k+1} (h).$
\end{proof}

\begin{proposition}
  Let  $(H, m, \eta, \a,  \Delta, \varepsilon, \b, S)$ be a Hom-Hopf algebra and $h\in H$ a group-like element, i.e, $\Delta(h)= h\ot h$. Then there exists $k\in \mathbb{ N}$ where
  \begin{equation}
    \a^{k} (S(h)h)=  \a^k ( h S(h))=1.
  \end{equation}
\end{proposition}
\begin{proof}
  There exists $k\in \mathbb{N}$ such that  $\a^{k}(S(h^{(1)}) h^{(2)}) = \a^{k'} (h^{(1)} S(h^{(2)}))= \varepsilon(x)1.$ Then

  $$\a^k(S(h)h)= \a^k( hS(h))=\varepsilon(h)1.$$
  Now the result is followed by the fact that $\varepsilon(h)=1$.
\end{proof}
By previous Proposition  the relative Home inverse of a group-like element $h$ is $S(h)$.

\end{document}